\documentclass[11pt]{amsart}
\usepackage{}
\usepackage{amssymb,amsmath,amsfonts,amsthm,hyperref,here,enumerate,xfrac,mathtools,graphicx,pstool}
\usepackage{amssymb,mathtools,graphics,xcolor,hyperref}
\definecolor{cerulean}{rgb}{0.0, 0.48, 0.65}
\hypersetup{
    colorlinks=true,       
    linkcolor=cerulean,          
    linkbordercolor=cerulean,
    citebordercolor=cerulean,
    citecolor=cerulean,        
    filecolor=cerulean,      
    urlcolor=cerulean           
}

\newcommand{\R}{\mathbb R}
\newcommand{\N}{\mathbb N}

\renewcommand{\d}{\mathrm{d}}
\DeclareMathOperator{\sinc}{sinc}
\DeclareMathOperator{\arccot}{arccot}
\newtheorem{theorem}{Theorem}[section]

\newtheorem{lemma}[theorem]{Lemma}
\newtheorem{prop}[theorem]{Proposition}
\theoremstyle{remark}
\newtheorem{example}[theorem]{Example}
\newtheorem{remark}[theorem]{Remark}

\title[Legendrian approximation of curves]{Convex Integration and Legendrian Approximation of Curves}
\author[N.~Hungerb\"uhler, T.~Mettler \& M.~Wasem]{Norbert Hungerb\"uhler, Thomas Mettler and Micha Wasem}
\address{Department of Mathematics, ETH Z\"urich, R\"amistrasse 101,
    8092 Z\"urich, Switzerland}
    \email{norbert.hungerbuehler@math.ethz.ch}
    \email{micha.wasem@math.ethz.ch}
\address{Institute for Mathematics, Goethe University Frankfurt, Robert-Mayer-Str.~10, 60325 Frankfurt am Main, Germany}\email{mettler@math.uni-frankfurt.de}

\date{April 30, 2015}

\begin{document}

\begin{abstract}
Using convex integration we give a constructive proof of the well-known fact that every continuous curve in a contact $3$-manifold can be approximated by a Legendrian curve. 
\end{abstract}
\maketitle
\section{Introduction}

A~\textit{contact structure} on a $3$-manifold $M$ is a maximally non-integrable rank $2$ subbundle $\xi$ of the tangent bundle of $M$. If $\alpha$ is a $1$-form on $M$ whose kernel is $\xi$, then $\xi$ is a contact structure if and only if $\alpha\wedge\d\alpha\neq 0$. A curve $\eta$ in a contact $3$-manifold $(M,\xi)$ is called~\textit{Legendrian}, whenever $\eta^*\alpha=0$ for some (local) $1$-form $\alpha$ defining $\xi$. \\

The purpose of this note is to give a detailed proof of the following statement which is often used in contact geometry and Legendrian knot theory.

\begin{theorem}\label{main}
Any continuous map from a compact $1$-manifold to a contact $3$-manifold  can be approximated by a Legendrian curve in the $C^{0}$-Whitney topology.
\end{theorem}
Whereas this theorem is a special case of Gromov's $h$-principle for Legendrian immersions~\cite{gromov}, the curve-case can be treated by more elementary techniques. Sketches of proofs of Theorem~\ref{main} have already appeared in the literature, see for example \cite[p.6-7]{etnyre}, \cite[p.40]{geiges} or \cite[p.102]{geiges2}. Exploiting the fact that every contact $3$-manifold is locally contactomorphic to $\R^3$ equipped with the standard contact structure defined by $\alpha=\d z -y \d x$, Etnyre and Geiges indicate that either the \emph{front-projection} $(x,z)$ of a given curve $(x,y,z)$ can be approximated by a zig-zag-curve whose slope approximates the $y$-component of the curve   or the \emph{Lagrangian projection} $(x,y)$ can be approximated by a curve whose area integral approximates the $z$ component of the curve, which can be achieved by adding small negatively or positively oriented loops.\\

Here, we give a different and analytically rigorous proof of Theorem~\ref{main} by using convex integration. Our proof has the advantage of providing a constructive approximation. In particular, in the case of a continuous curve in $\R^3$ equipped with the standard contact structure, we obtain an explicit Legendrian curve given in terms of an elementary integral.    For instance, we obtain an explicit solution to the ``parallel parking problem'' in Example~\ref{parallel}. Example~\ref{helix} shows how our technique recovers the zig-zag-curves and the small loops in the front - respectively Lagrangian projections.

\section{Proof of the Theorem}

We start by first treating the case where the contact manifold is $\R^3$ equipped with the standard contact structure, that is, we aim to prove the following:
\begin{prop}\label{ruedi}
Let $\upsilon\in C^0([0,2\pi],\R^3)$. For every $\varepsilon>0$ there exists a Legendrian curve $\eta\in C^\infty([0,2\pi],\R^3)$ such that $\|\upsilon-\eta\|_{C^0([0,2\pi])}\leqslant\varepsilon$. 
\end{prop}
\begin{remark}
Here, as usual, $\|\gamma\|_{C^0(I)}\coloneqq \sup_{t\in I}|\gamma(t)|$ and $\|\gamma\|_{C^1(I)}\coloneqq \|\gamma\|_{C^0(I)}+\|\d \gamma\|_{C^0(I)}$.
\end{remark}

Let the curve we wish to approximate be given by $(x,y,z)\in C^\infty([0,2\pi],\R^3)$. The regularity is no restriction due to a standard approximation argument using convolution. Let $\eta=(a,b,c) \in C^\infty([0,2\pi],\R^3)$ denote the approximating Legendrian curve. For every choice of smooth functions $(a,c) \in C^\infty([0,2\pi],\R^2)$ satisfying $\dot{a}\neq 0$, we obtain a Legendrian curve by defining $b=\dot{c}/\dot{a}$. Therefore, if $(\dot{a}(t),\dot{c}(t))$ lies in the set 
$$
\mathcal R_{t,\varepsilon}\coloneqq\left\{(u,v)\in\R^2,|v-y(t)u|\leqslant \varepsilon \min\{|u|,|u|^2\}\right\},
$$
for every $t \in [0,2\pi]$, then $\|b-y\|_{C^{0}([0,2\pi])}\leqslant \varepsilon$. This condition can be achieved by defining  
$$
(a(t),c(t)):=(x(0),z(0))+\int_0^t \gamma(u,nu)\,\mathrm du,
$$
with $\gamma\in C^\infty([0,2\pi]\times S^1,\R^2)$ and $n \in \mathbb{N}$, provided that $\gamma(t,\cdot)\in \mathcal R_{t,\varepsilon}$. Furthermore, if $\gamma$ additionally satisfies
$$
\frac{1}{2\pi}\oint_{S^1}\gamma(t,s)\,\mathrm ds = (\dot x(t), \dot z(t)),
$$
for all $t \in [0,2\pi]$, then --  as we will show below -- $(a(t),c(t))$ approaches $(x(t),z(t))$ as $n$ gets sufficiently large. \\

The set $\mathcal R_{t,\varepsilon}$ is \emph{ample}, i.e., the interior of its convex hull is all of $\R^2$. For any given point $(\dot x(t), \dot z(t))\in\R^2$ we will thus be able to find a loop in $\mathcal R_{t,\varepsilon}$ having $(\dot x(t), \dot z(t))$ as its barycenter. This fact is sometimes referred to as the fundamental lemma of convex integration (see for instance~\cite[Prop.~2.11, p.~28]{MR3024860}). In the particular case studied here we obtain an explicit formula for $\gamma$: 
\begin{lemma}\label{loops}
There exists a family of loops $\gamma\in C^\infty([0,2\pi]\times S^1,\R^2)$ satisfying $\gamma(t,\cdot)\in \mathcal R_{t,\varepsilon}$ and such that
\begin{equation}\label{barycenter}
\frac{1}{2\pi}\oint_{S^1}\gamma(t,s)\,\mathrm ds = (\dot x(t), \dot z(t)), 
\end{equation}
for all $t \in [0,2\pi]$. 
\end{lemma}

\begin{proof}
The map $\gamma\coloneqq(\gamma_1,\gamma_2)$, where 
$$
\gamma_1(t,s)\coloneqq r \cos s + \dot x(t)
$$ 
and
$$
\gamma_2(t,s):= \gamma_1(t,s)\left(y(t)+\frac{2(\dot z(t)-y(t)\dot x(t))}{r^2+2\dot x(t)^2}\gamma_1(t,s)\right)
$$
satisfies \eqref{barycenter} for every $r>0$. If $r$ is large enough one obtains $\gamma(t,\cdot)\in \mathcal R_{t,\varepsilon}$, where $r$ can be chosen independently of $t$ by compactness of $[0,2\pi]$. 
\end{proof}
We now have:
\begin{proof}[Proof of Proposition \ref{ruedi}]
With the definitions above we obtain
$$b(t):=\frac{\dot{c}(t)}{\dot{a}(t)}=y(t)+\frac{2(\dot z(t)-y(t)\dot x(t))}{r^2+2\dot x(t)^2}\gamma_1(t,nt).$$
We are left to show that $\left|(a, c)-(x,z)\right|\leqslant \varepsilon$ provided $n$ is large enough. This follows from the following estimate
\begin{equation}
\left\|(a, c)-(x,z)\right\|_{C^0([0,2\pi])}\leqslant \frac{4\pi^2}{n}\|\gamma\|_{C^1([0,2\pi]\times S^1)}\label{ac}.
\end{equation}
The estimate is in fact a geometric property of the derivative and can be interpreted as follows: Since $(\dot{a},\dot{c})$ and $(\dot x, \dot z)$ coincide ``in average'' on shorter and shorter intervals when $n$ gets bigger and bigger, $(a,c)$ and $(x,z)$ tend to become close: Let
$$I_k:=\left[\frac{2\pi k}{n},\frac{2\pi(k+1)}{n}\right]\text{ for }k=0,\ldots, \left\lfloor \frac{nt}{2\pi}\right\rfloor-1\text{ and }J\coloneqq\left[\left\lfloor \frac{nt}{2\pi}\right\rfloor\frac{2\pi}{n},t\right].$$
Then we can estimate $D=\left|(a(t), c(t))-(x(t),z(t))\right|$:
$$\begin{aligned}
D  = &\left|\int_0^t\gamma(u,nu)\,\mathrm du-\int_0^t (\dot x,\dot z)(u)\,\mathrm du\right|\\
 \leqslant &\sum_{k=0}^{\left\lfloor \frac{nt}{2\pi}\right\rfloor-1}\left|\int_{I_k}\gamma(u,nu)\,\mathrm du - \int_{I_k}\frac{1}{2\pi}\int_0^{2\pi}\gamma(u,v)\,\mathrm dv\,\mathrm du\right|+\\
 &+\int_{J}\left(\left|\gamma(u,nu)\right|+\|\gamma\|_{C^0([0,2\pi]\times S^1)}\right)\,\mathrm du \\
 \leqslant &\sum_{k=0}^{\left\lfloor \frac{nt}{2\pi}\right\rfloor-1}\left|\frac{1}{n}\int_0^{2\pi}\gamma\left(\frac{v+2k\pi}{n},v\right)\,\mathrm dv - \int_{I_k}\frac{1}{2\pi}\int_0^{2\pi}\gamma(u,v)\,\mathrm dv\,\mathrm du\right|+\\
 &+\frac{4\pi}{n}\|\gamma\|_{C^0([0,2\pi]\times S^1)}\\
 \leqslant &\sum_{k=0}^{\left\lfloor \frac{nt}{2\pi}\right\rfloor-1}\left|\frac{1}{2\pi} \int_{I_k}\int_0^{2\pi}\left(\gamma\left(\frac{v+2k\pi}{n},v\right)-\gamma(u,v)\right)\,\mathrm dv\,\mathrm du\right|+\frac{4\pi}{n}\|\gamma\|_{C^0([0,2\pi]\times S^1)}\\
 \leqslant &\left\lfloor \frac{nt}{2\pi}\right\rfloor\frac{4\pi^2}{n^2}\|\partial_t\gamma\|_{C^0([0,2\pi]\times S^1)}+\frac{4\pi}{n}\|\gamma\|_{C^0([0,2\pi]\times S^1)}\\
 \leqslant &\frac{4\pi}{n}\left(\pi\|\partial_t\gamma\|_{C^0([0,2\pi]\times S^1)}+\|\gamma\|_{C^0([0,2\pi]\times S^1)}\right).
\end{aligned}$$
By construction, the curve $(a,b,c)$ is Legendrian and an approximation of $(x,y,z)$, provided $n$ is large enough.\end{proof}
Next we show that we can approximate closed curves by closed Legendrian curves.
\begin{prop}
Let $\upsilon\in C^0(S^1,\R^3)$. For every $\varepsilon>0$ there exists a Legendrian curve $\eta\in C^\infty(S^1,\R^3)$ such that $\|\upsilon-\eta\|_{C^0(S^1)}\leqslant\varepsilon$. 
\end{prop}
\begin{proof} Using standard regularization, let the curve we wish to approximate be given by $(x,y,z)\in C^\infty([0,2\pi],\R^3)$, where the values of $(x,y,z)$ in $0$ and $2\pi$ agree to all orders. Define $g(t)\coloneqq \gamma_1^2(t,nt)$. Since  $\|g\|_{L^1([0,2\pi])}=O(r^2)$ as $r\to\infty$, we can choose $r>0$ large enough such that $f\coloneqq g/\|g\|_{L^1([0,2\pi])}$ is well-defined. With the notation
$$\mathrm I_2\coloneqq \int_0^{2\pi}\gamma_2(u,nu)\,\mathrm du,$$
we define $\eta=(a,b,c)$ as follows:
\begin{align}
(a(t),c(t)) & \coloneqq (x(0),z(0)) + \int_0^t \bigg[\gamma(u,nu)-(0,\mathrm I_2f(u))\bigg]\,\mathrm du,\label{c}\\
b(t) & \coloneqq \frac{\dot c(t)}{\dot a(t)} = y(t) + \gamma_1(t,nt)\left(\frac{2(\dot z(t)-y(t)\dot x(t))}{r^2+2\dot x(t)^2}-\frac{\mathrm I_2}{\|g\|_{L^1([0,2\pi])}}\right).\label{b}
\end{align}
A straightforward computation shows that the values of $(a,b,c)$ in $0$ and $2\pi$ agree to all orders, hence $\eta\in C^\infty(S^1,\R^3)$ and it is Legendre by construction. Using \eqref{ac} we obtain $|\mathrm I_2|\leqslant \frac{4\pi^2}{n}\|\gamma_2\|_{C^1([0,2\pi]\times S^1)}$, hence we find using \eqref{b} as $r\to\infty$:
$$
\begin{aligned}
\|b-y\|_{C^0([0,2\pi])} & \leqslant \|\gamma_1\|_{C^0([0,2\pi]\times S^1)}\left(1+\frac{1}{n}\|\gamma\|_{C^1([0,2\pi]\times S^1)}\right)O(r^{-2}).
\end{aligned}
$$
For the remaining components we find find using \eqref{ac} and \eqref{c} the uniform bound
$$\begin{aligned}
|(a(t),c(t))-(x(t),z(t))| & \leqslant \frac{4\pi^2}{n}\|\gamma\|_{C^1([0,2\pi]\times S^1)} + \frac{\left|\mathrm I_2\right|}{\|g\|_{L^1([0,2\pi])}}\int_0^t g(u)\,\mathrm du
\\&\leqslant \frac{8\pi^2}{n}\|\gamma\|_{C^1([0,2\pi]\times S^1)}.\end{aligned}$$
Choosing $r$ large enough and $n\sim r^2$ concludes the proof.
\end{proof}
We show now how to glue together two local approximations of a curve $\Gamma$ in $M$ on two intersecting coordinate neighborhoods. Let therefore $U_\sigma$ and $U_\tau$ in $M$ be coordinate patches such that $U=U_\sigma\cap U_\tau\ne\emptyset$. Let $I_\sigma$ and $I_\tau$ be compact intervals such that $I=I_\sigma\cap I_\tau$ contains an open neighborhood of $t=0$ (after shifting the variable $t$ if necessary) and such that $\Gamma(I_\sigma)\subset U_\sigma$, $\Gamma(I_\tau)\subset U_\tau$. Assume without restriction that $\Gamma$ is smooth and let $(x,y,z)$ represent $\Gamma$ on $U$. Suppose that $(x,y,z)$ is approximated by Legendrian curves $\sigma:I_\sigma\to\R^3$ and $\tau:I_\tau\to \R^3$ such that\begin{equation}
\|\sigma-(x,y,z)\|_{C^0(I)}<\varepsilon^2,\quad \|\tau-(x,y,z)\|_{C^0(I)}<\varepsilon^2\label{sigmatauclose}
\end{equation}
for some fixed $0<\varepsilon<\frac{1}{2}$. For $r>0$, define $R(r)$ to be the smallest number such that $\bar B_r(0)\subset\operatorname{conv}\left(\mathcal R_{0,\varepsilon}\cap \bar B_{R}(0)\right)$. Note that $R$ depends continuously on $r$ and if $r>r_0\coloneqq\varepsilon/\sqrt{1+y(0)^2}$, then
\begin{equation}
R(r) = \frac{r}{\varepsilon}\sqrt{(1+y(0)^2)\left(1+(|y(0)|+\varepsilon)^2\right)}\eqqcolon\frac{r}{\varepsilon}w(y(0),\varepsilon)\label{formulaforR}.
\end{equation}
Choose $0<\delta<\varepsilon^2$ such that $[-\delta,\delta]\subset I$ and such that $\delta\|(x,y,z)\|_{C^1(I)}\leqslant \varepsilon^2$
and define
$$\begin{aligned}
p_1&\coloneqq(\sigma_1(-\delta),\sigma_3(-\delta)),\\
\dot p_1&\coloneqq(\dot\sigma_1(-\delta),\dot\sigma_3(-\delta)),\\
 p_2&\coloneqq(\tau_1(\delta),\tau_3(\delta)),\\
\dot p_2&\coloneqq(\dot\tau_1(\delta),\dot\tau_3(\delta)).
\end{aligned}$$
From \eqref{sigmatauclose} and the choice of $\delta$ we obtain $\dot p_1,\dot p_2\in \mathcal C_{\varepsilon}\coloneqq\big\{(u,v)\in\R^2,|v-y(0)u|\leqslant\varepsilon|u|\big\}$ and
$$
\frac{p_2-p_1}{2\delta}\eqqcolon  p\in B_{\bar r}(0),\text{where }\bar r=\frac{2\varepsilon^2}{\delta}.$$
Since $3\bar r>r_0$, we can express $R(3\bar r)$ by means of formula \eqref{formulaforR}. This will be used in computation \eqref{gianni29}. 
We construct a path $\gamma=(\gamma_1,\gamma_2):[-\delta,\delta]\to \mathcal C_{\varepsilon}$ as follows: For $\rho<\delta/2$, let $\gamma|_{[-\delta,-\delta+\rho]}$ be a continuous path from $\dot p_1$ to $0$ and let $\gamma|_{[\delta-\rho,\delta]}$ be a continuous path from $0$ to $\dot p_2$. We construct $\gamma$ such that the quotient $\gamma_2/\gamma_1$ is well-defined on $[-\delta,-\delta+\rho]\cup [\delta-\rho,\delta]$ and equals $y(0)$ in $t=-\delta+\rho$ and $t=\delta-\rho$. Moreover, we require that
\begin{equation}\int_{-\delta}^{-\delta+\rho}|\gamma(t)|\,\mathrm dt < \frac{\delta\varepsilon}{2}\text{ and }\int_{\delta-\rho}^{\delta}|\gamma(t)|\,\mathrm dt < \frac{\delta\varepsilon}{2}.\label{paths}\end{equation}
On $[-\delta,-\delta+\rho],$ such a path is for example given by
$$
t\mapsto \left(1-\frac{\delta + t}{\rho}\right)^k\begin{pmatrix}\dot\sigma_1(-\delta) \\ y(0)\dot\sigma_1(-\delta) + (\dot\sigma_3(-\delta) - y(0)\dot\sigma_1(-\delta))\left(1-\frac{\delta + t}{\rho}\right)^{k}\end{pmatrix}
$$
provided $k\in \N$ is sufficiently large. We obtain
$$
\frac{1}{2(\delta-\rho)}\left(2\delta p-\int_{-\delta}^{-\delta+\rho}\gamma(t)\,\mathrm dt-\int_{\delta-\rho}^{\delta}\gamma(t)\,\mathrm dt\right)\eqqcolon \bar p\in B_{3\bar r}(0)
$$
and hence $\bar p\in \operatorname{int}\operatorname{conv}(B_{R(3\bar r)}(0)\cap\mathcal R_{0,\varepsilon})$. Using the fundamental lemma of convex integration we let $\gamma|_{[-\delta+\rho,\delta-\rho]}$ be a continuous closed loop in $B_{R(3\bar r)}(0)\cap\mathcal R_{0,\varepsilon}$ based at $0$ such that
$$
\frac{1}{2(\delta-\rho)}\int_{-\delta+\rho}^{\delta-\rho}\gamma(t)\,\mathrm dt = \bar p.$$
With these definitions we obtain
$$
\frac{1}{2\delta}\int_{-\delta}^{\delta}\gamma(t) = p.
$$
Now we define $\eta=(a,b,c):[-\delta,\delta]\to\R^3$ by letting $b(t)\coloneqq \dot c(t)/\dot a(t)$, where
$$\begin{aligned}
(a,c)(t) & \coloneqq p_1+\int_{-\delta}^t \gamma(u)\mathrm du.
\end{aligned}$$
The curve $\eta$ is well-defined and Legendrian by construction. It satisfies $\eta(-\delta)=\sigma(-\delta)$ and $\eta(\delta)=\tau(\delta)$. Moreover, $(a,c)$ and $(\sigma_1,\sigma_3)$ agree to first order in $t=-\delta$ and so do $(a,c)$ and $(\tau_1,\tau_3)$ in $t=\delta$. From $\gamma([-\delta,\delta])\in \mathcal C_\varepsilon$ and the choice of $\delta$ we find
\begin{equation}\label{luxusiglu}
|b(t)-y(t)|\leqslant|b(t)-y(0)|+|y(t)-y(0)|\leqslant \varepsilon + \delta\|y\|_{C^1(I)}<2\varepsilon.
\end{equation}
Using \eqref{sigmatauclose}, \eqref{formulaforR}, \eqref{paths} and the choice of $\delta$ we obtain for the remaining components the uniform bound
\begin{equation}
\begin{aligned}\label{gianni29}
|(a,c)(t)-(x,z)(t)| & \leqslant |p_1-(x,z)(-\delta)|+\int_{-\delta}^t \left(|\gamma(u)|+|(\dot x,\dot z)(u)|\right)\,\mathrm du\\
& \leqslant \varepsilon^2 + \delta\varepsilon +\int_{-\delta+\rho}^{\delta-\rho}|\gamma(u)|\,\mathrm du + 2\delta\|(x,z)\|_{C^1(I)}\\
& \leqslant 2\varepsilon + 2\delta R(3\bar r) \\
& \leqslant \varepsilon \left(14+12\left(|y(0)|+\frac{1}{2}\right)^2\right).
\end{aligned}
\end{equation}
Finally, suppose $\upsilon$ is a continuous curve from a compact $1$-manifold $N$ (that is, $N$ is a compact interval or $S^1$) into a contact $3$-manifold $(M,\xi)$. We fix some Riemannian metric $g$ on $M$. Then it follows with the bounds~(\ref{luxusiglu},\ref{gianni29}) and the compactness of the domain of $\upsilon$ that for every $\varepsilon>0$ there exists a $\xi$-Legendrian curve $\eta$ such that  
$$
\sup_{t \in N} d_{g}(\upsilon(t),\eta(t))<\varepsilon, 
$$
where $d_{g}$ denotes the metric on $M$ induced by the Riemannian metric $g$. In particular, every open neighborhood of $\upsilon \in C^0(N,M)$ -- equipped with the uniform topology -- contains a Legendrian curve $N \to M$. Since $N$ is assumed to be compact the uniform topology is the same as the Whitney $C^{0}$-topology, thus proving Theorem~\ref{main}. 

\section{Examples}
\begin{example}[Parallel Parking]\label{parallel}
The trajectory of a car moving in the plane can be thought of as a curve $[0,2\pi]\to S^1\times \R^2$. Denoting by $(\varphi,a,c)$ the natural coordinates on $S^1\times \R^2$, the angle coordinate $\varphi$ denotes the orientation of the car with respect to the $a$-axis and the coordinates $(a,c)$ the position of the car in the plane. Admissible motions of the car are curves satisfying 
$$
\dot a\sin \varphi = \dot c \cos \varphi.
$$ 
The manifold $S^1\times \R^2$ together with the contact structure defined by the kernel of the $1$-form $\theta:=\sin \varphi\,\d a - \cos \varphi\, \d c$ is a contact 3-manifold. Indeed, we have
$$
\theta\wedge\d\theta=-\cos^2\!\varphi\,\d\varphi\wedge\d a \wedge \d c-\sin^2\!\varphi\,\d\varphi\wedge\d a \wedge \d c=-\d\varphi\wedge\d a \wedge \d c\neq 0. 
$$
Applying Theorem~\ref{main} with $b=\tan\varphi$ gives an explicit approximation of the curve 
$$
t\mapsto (x(t),y(t),z(t))=(0,0,t).
$$ 
Lemma~\ref{loops} gives the loop 
$$
\gamma(t,s)=2(r \cos s,\cos^2s),
$$ 
and hence the desired Legendrian curve 
$$
\left(\arccot(r \sec (n t)),2 r t \sinc(n t),t + t\sinc(2 n t)\right),
$$
provided $r$ is large enough and $n\sim r^2$.
\end{example}
\begin{figure}
\begin{center}
\includegraphics[scale=0.45]{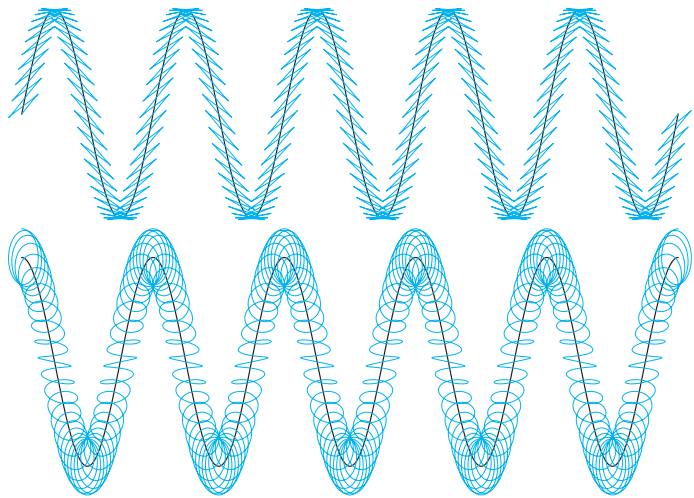}
\caption{The front (top) and the Lagrangian projection (bottom) of the Legendrian approximation of $\upsilon$.}\label{642}
\end{center}
\end{figure}
\begin{example}[Legendrian Helix]\label{helix}
The Legendrian approximation of the helix 
$$
\upsilon : [0,2\pi] \to \R^3, \quad t\mapsto (t,\cos (5t),\sin(5t)),
$$
with $n=\frac{2}{9}r^2$ and $r=30$ is given by
\begin{align*}
a(t) =& ~ t + \frac{3}{20} \sin(200 t)\\ b(t) =&~ \frac{455}{451} \cos(5 t) +\frac{120}{451}\cos(5t) \cos(200 t) \\ c(t)  =&~
\sin(5 t) + \frac{459}{5863} \sin(195 t) + \frac{1377}{18491} \sin(205 t)+\frac{180}{35629} \sin(395 t)+\\ &~ + \frac{20}{4059}\sin(405 t).
\end{align*}
and produces the zig-zags and the small loops in its front and Lagrangian projections (see Figure \ref{642}).
\end{example}

\providecommand{\bysame}{\leavevmode\hbox to3em{\hrulefill}\thinspace}
\providecommand{\MR}{\relax\ifhmode\unskip\space\fi MR }
\providecommand{\MRhref}[2]{%
  \href{http://www.ams.org/mathscinet-getitem?mr=#1}{#2}
}
\providecommand{\href}[2]{#2}

\end{document}